\documentclass{amsart}

\usepackage{amsmath,tabu}
\usepackage{amsfonts}
\usepackage{amssymb}
\usepackage{amsthm}
\usepackage[bookmarks]{hyperref}
\hypersetup{colorlinks=false}
\usepackage{cleveref}
\usepackage{setspace}
\usepackage{enumerate}
\usepackage{amscd}
\usepackage{tikz-cd}

\newcommand{\mS}{\mathcal{S}}
\newcommand{\mT}{\mathcal{T}}

\newcommand{\ot}{\otimes}
\newcommand{\ra}{\rightarrow}
\newcommand{\C}{\mathbb{C}}

\newcommand{\N}{\mathbb{N}}
\newcommand{\F}{\mathbb{F}}
\newcommand{\mC}{\mathcal{C}}
\newcommand{\D}{\mathcal{D}}

\newcommand{\mO}{\mathcal{O}}

\newcommand{\ep}{\varepsilon}

\newcommand{\mF}{\mathfrak{u}}

\newcommand{\ilim}{\displaystyle\lim_{\longrightarrow}}

\theoremstyle{thmit}
\newtheorem{theorem}{Theorem}[section]
\newtheorem{definition}[theorem]{Definition}
\newtheorem{proposition}[theorem]{Proposition}

\newtheorem{corollary}[theorem]{Corollary}
\numberwithin{equation}{section}

\theoremstyle{definition}
\newtheorem{example}[subsubsection]{Example}
\newtheorem{remarks}[theorem]{Remark}

\title[Inductive limit]{\textbf{Nuclearity properties and $C^*$-envelopes of operator system inductive limits}}
\date{\today}

\author[Preeti Luthra]{Preeti Luthra}
\address{Department of Mathematics\\ University of Delhi\\ Delhi-110007, INDIA}
\email{maths.preeti@gmail.com}

\author[Ajay Kumar]{Ajay Kumar}
\address{Department of Mathematics\\ University of Delhi\\ Delhi-110007, INDIA}
\email{akumar@maths.du.ac.in}
\begin{document}
\maketitle
\begin{abstract}
We study the relationship between $C^*$-envelopes and inductive limit of operator systems. Various operator system nuclearity properties of inductive limit for a sequence of operator systems are also discussed.
\end{abstract}
\section{Introduction}
In last few years, the development of the theory of operator systems has seen a fair amount of attention. All the important notions from the theory of $C^*$-algebras including exactness, nuclearity, weak expectation property and lifting properties have been explicitly defined in the category of operator systems. Associated to every representation $\phi$ of an operator system  $\mS$ into $B(H)$, for some Hilbert Space $H$, there exist a $C^*$-cover the $C^*$-algebra $C^*(\phi(\mS))\subset B(H)$. The minimal $C^*$-cover among all such representation is known as the $C^*$-envelope of $\mS$. It is thus quite natural to ask which $C^*$-algebraic properties of the $C^*$-envelopes are carried over to the generating operator systems in terms of their definitions in the operator system category, and to what extent. Some attempts done in this direction can be found in \cite{opsysnuc, embopsys}. \\
It is well known (see \cite{blackadar2}) that for the category of $C^*$-algebras, inductive limit preserves many intrinsic properties, viz., exactness, nuclearity, simplicity etc. The analysis of inductive limit of ascending sequences of finite dimensional $C^*$-algebras, known as approximately finite dimensional (AF) $C^*$-algebras, has played an important role in theory of operator algebras. 
Existence of inductive limits in the category of operator systems has been shown in \cite{caralgebra}. But unlike in the category of $C^*$-algebras, there are several notion of nuclearity in the operator system category (see \cite{KPTT1, KPTT2, kavnuc}). It is thus natural to check if these nuclearity properties are preserved under operator system inductive limit. This paper deals with $C^*$-envelopes of operator system inductive limit and various nuclearity properties of operator system inductive limit.\\
After recalling all the prerequisite definitions and results including that of inductive limit for the categories of $C^*$-algebras and operator systems in \Cref{s:1} on Preliminaries, the commutativity of inductive limit and  $C^*$-envelopes has been studied in \Cref{s:2}. We do this for particular cases of direct sequences. In \Cref{s:3}, various nuclearity properties of inductive limit of operator systems are discussed.

\section{Preliminaries}\label{s:1}
\subsection{Operator systems}
The concept of operator systems and their tensor products is the familiar one now and most of the details can be seen in \cite{kavnuc, KPTT1, KPTT2}.\\
 Recall that a concrete operator system is a unital self-adjoint subspace of $B(H)$ for some Hilbert space $H$. A \textit{$C^*$-cover} (\cite[$\S 2$]{hamana2}) of an operator system
$\mS$ is a pair $(A, i)$ consisting of a unital $C^*$-algebra $A$ and
a complete order embedding $i : \mS \rightarrow A$ such that $i(A)$
generates the $C^*$-algebra $A$. The \textit{$C^*$-envelope} as defined by Hamana \cite{hamana2}, of an operator system $\mS$ is a
$C^*$-cover generated by $\mS$ in its
injective envelope $I(\mS)$ and is denoted by $C^*_e (\mS)$.  The $C^*$-envelope $C^*_e (\mS)$
enjoys the following universal ``minimality'' property (\cite[Corollary 4.2]{hamana2}):

{\em Identifying $\mS$ with its image in $C^*_e (\mS)$, for any $C^*$-cover
$(A, i)$ of $\mS$, there is a unique surjective unital
$*$-homomorphism $\pi : A \ra C^*_e (\mS)$ such that $\pi(i(s)) = s$
for every $s$ in $\mS$.}\\
One more
fundamental $C^*$-cover, the maximal one, is associated to an operator system $\mS$,
namely, the universal $C^*$-algebra $C^*_u(\mS)$ introduced by
Kirchberg and Wassermann (\cite[$\S 3$]{kirch}). $C^*_u(\mS)$ satisfies the following universal ``maximality'' property:\\
\emph{Every unital completely positive map $\phi: \mS \ra A$, where $A$ is a unital $C^*$-algebra extends uniquely to a unital $*$-homomorphism $\pi : C^*_u(\mS) \ra A$.}\\
In general, a complete order isomorphism from an operator system to another may not extend to an injective $*$-homorphism on their $C^*$-covers.
\begin{proposition}\label{extcover} For operator systems $\mS$ and $\mT$, let $\theta : \mS \ra \mT$ be a unital complete order isomorphism. Then
\begin{enumerate} 
\item[(i)]\cite[Proposition 9]{kirch} $\theta$ can be uniquely extended to a unital $*$-monomorphism $\overline{\theta} : C_u^*(\mS) \ra C_u^*(\mT)$, that is, $\overline{\theta} \circ i_{\mS}=i_{\mT} \circ \theta$; where  $i_{\mS}: \mS \ra C^*_u(\mS)$ and $i_{\mT}: \mT \ra C^*_u(\mT)$ denote the natural inclusion map.
\item[(ii)]\cite[Theorem 2.2.5]{Arveson} In case $\theta$ is surjective, then $\theta$ can be uniquely extended to a unital $*$-monomorphism $\overline{\theta} : C_e^*(\mS) \ra C_e^*(\mT)$, that is, $\overline{\theta} \circ i_{\mS}=i_{\mT} \circ \theta$; where  $i_{\mS}: \mS \ra C^*_e(\mS)$ and $i_{\mT}: \mT \ra C^*_e(\mT)$ denote the natural inclusion map.
\end{enumerate}
\end{proposition}
A lattice of tensor products of
operator systems admitting a natural partial order: 
 $$\min \leq \mathrm{e} \leq \mathrm{el} ,\mathrm{er} \leq  \mathrm{c} \leq \max,$$ were introduced in \cite{KPTT1}. In \cite{disgrp}, a natural operator system
tensor product ``$\mathrm{ess}$'' arising from the enveloping $C^*$-algebras,
viz., $\mathrm{\mS \ot_{ess} \mT \subseteq C^*_e(\mS) \ot_{max} C^*_e(\mT)},$
was also defined. It is known from \cite[$\S 8$]{disgrp} that $\mathrm{ess} \leq \mathrm{c}$. See also \cite[Proposition 4.4]{opsysnuc} for comparison of $\mathrm{ess}$ with other operator system tensor products.
Given two operator system tensor
products $\alpha$ and $\beta$, an operator system $\mS$ is said to be
\emph{$(\alpha, \beta)$-nuclear} if the identity map between $\mS
\ot_{\alpha} \mT$ and $\mS \ot_{\beta} \mT$ is a complete order
isomorphism for every operator system $\mT$, i.e. $$\mS \ot_{\alpha}
\mT = \mS \ot_{\beta} \mT.$$ Also, an operator system $\mS$ is said to
be \textit{$C^*$-nuclear}, if $$\mS \ot_{\mathrm{min}} A =\mS
\ot_{\mathrm{max}} A$$ for all unital $C^*$-algebras $A$.\\

Given an operator system $\mS$, a subspace $J\subseteq \mS$ is said to
be a \textit{kernel} (\cite[Definition $3.2$]{KPTT2}) if there exists
an operator system $\mT$ and a unital completely positive map $\phi :
\mS \ra \mT$ such that $J = \ker \phi$. For such a kernel $J \subset
\mS$, Kavruk et al. have shown that the quotient space $\mS/J$ forms an
operator system (\cite[Proposition $3.4$]{KPTT2}) with respect to the
natural involution, whose positive cones are given by 
\begin{eqnarray*}
  \mC_n &=& \mC_n(\mS/J)\\ &=& \{(s_{ij} + J)_{i,j} \in M_n(\mS/J):\;
  (s_{ij}) + \ep(1 + J)_n \in \D_n \;\text{, for every} \; \ep >
  0\},
  \end{eqnarray*} 
  where $$\D_n=\{(s_{ij}+J)_{i,j} \in M_n(\mS/J):
\;(s_{ij}+y_{ij})_{i,j} \in M_n(\mS)^+\; \text{, for some $y_{ij} \in
  J$}\}$$ and the Archimedean unit is the coset $1 + J$.\\
 An operator system $\mS$ is said to be 
 $\mathrm{exact}$ (\cite{KPTT2})  if 
for every unital $C^*$-algebra $A$ and a closed ideal $I$ in $A$ the induced map
$$(\mS \hat{\ot}_{\mathrm{min}} A)/(\mS \bar{\ot} I) \ra \mS \hat{\ot}_{\mathrm{min}}
(A/ I)$$ is a complete order isomorphism.
 Recall that an operator subsystem $\mS$ of a unital $C^*$-algebra $A$ is said \textit{to
contain enough unitaries of $A$} if the unitaries in $\mS$ generate $A$
as a $C^*$-algebra (\cite[$\S 9$]{KPTT2}).

The following characterizations established in \cite[$\S 5$]{KPTT2}, \cite[$\S 4$]{kavnuc}, \cite{opsysnuc} and \cite{embopsys} are used quite often:
\begin{theorem}\label{opsysenv}
\begin{enumerate}
\item[(i)] An operator system with exact $C^*$-envelope is exact. Converse is true if the operator system contains enough unitaries in its $C^*$-envelope. 
\item[(ii)] An operator system is $(\min, \mathrm{ess})$-nuclear if its
$C^*$-envelope is nuclear. Moreover, a unital $C^*$-algebra is $(\min,
\mathrm{ess})$-nuclear as an operator system if and only if it is
nuclear as a $C^*$-algebra.
\item[(iii)] Let $\mS \subset A$ contain enough unitaries of the unital
$C^*$-algebra $A$. Then $\mS$ is $(\min, \mathrm{ess})$-nuclear if and
only if $A$ is a nuclear $C^*$-algebra.
\item[(iv)] Suppose $\mS \subset A$ contains enough unitaries. Then, upto a
$\ast$-isomorphism fixing $\mS$, we have $A = C^*_e(\mS)$.
\end{enumerate}
\end{theorem}

\subsection{Inductive limits}
For convenience, we recall some fundamental definitions and facts about 
inductive limits in the general context of a category \cite{rordamk}. \\
An inductive sequence in a category $C$ is a sequence $\{X_n\}_{n=1}^{\infty}$ of objects in $C$ and a sequence $\phi_n: X_n \ra X_{n+1}$ 
of morphisms in $C$,
$$\begin{CD}
{X_1} @>{\phi_1}>>   {X_2}  @>{\phi_2}>> {X_3} @>{\phi_3}>> \cdots 
\end{CD}$$
For $m > n$, the composed morphisms 
$$\phi_{m,n} = \phi_{m-1} \circ \phi_{m-2} \circ \cdots \circ \phi_n: X_n \ra X_m $$
which, together with the morphism $\phi_n$ , are called the \emph{connecting morphisms}. Conventionally, $\phi_{n,n}= Id_{X_n}$ and $\phi_{m,n}= 0$ when $m <n.$ \\An \emph{inductive limit} of the inductive 
sequence $\{X_n,\phi_n\}_{n=1}^{\infty}$
in a category $C$ is a system 
$\{X,\mu_n\}_{n=1}^{\infty}$, where $X$ is an object in $C$ and $\mu_n: X_n \ra X$ is a morphism in $C$ for each $n \in \mathbb{N}$ satisfying the following conditions:
\begin{enumerate}
\item[(i)] The diagram 
\begin{tikzcd}[column sep=small]
 X_n \arrow{dr}[swap]{\mu_n} \arrow{rr}{\phi_n} & & X_{n+1} \arrow{dl}{\mu_{n+1}} \\& X &
\end{tikzcd} commutes for each $n$.

\item[(ii)] If $(Y,\{\lambda_n\}_{n=1}^{\infty})$ is a system, where $Y$ is an object in $C$, $\lambda_n: X_n \ra Y$ is 
a morphism in $С$ for each $n \in \mathbb{N}$, and $\lambda_n = \lambda_{n+1}\circ \phi_n$ for all $n \in \N$, then there exist a unique $\lambda: X\ra Y$ such that  
\begin{tikzcd}[column sep=small]
 & X_n \arrow{dr}{\lambda_n} \arrow{dl}[swap]{\mu_{n}} &\\ X  \arrow{rr}{\lambda}  & & Y
\end{tikzcd} commutes for each $n$.

\end{enumerate} 
The inductive limit of the sequence is denoted by $\ilim(X_n,\phi_n)$, or 
more briefly by $\displaystyle\lim_{\longrightarrow}X_n$. 

\vspace{3mm}
Inductive limits do not exist in all categories (see \cite[Exercise $6.4$]{rordamk}). 
We deal with the inductive limits only in the categories of $C^*$-algebras and operator systems.

\subsubsection{Inductive limits in the category of $C^*$-algebras}
\cite[Proposition $6.2.4$]{rordamk} 

Every inductive sequence $\{A_n\}_{n=1}^{\infty}$ of the $C^*$-algebras with $\phi_n : A_n \ra A_{n+1}$, $n = 1, 2, 3, \cdots$ the connecting $*$-homomorphisms has an inductive limit satisfying \begin{enumerate}
\item[(i)] $A=\overline{\bigcup_{n=1}^{\infty}\mu_n( A_n)};$
\item[(ii)] $\|\mu_n(a)\|=\lim_{m \ra \infty}\|\phi_{m,n}(a)\|$ for all $n \in \N$ and $a \in A.$
\item[(iii)] If $(B,\{\lambda_n\}_{n=1}^{\infty})$ is an inductive system with $\lambda : A \ra B$ as in (ii) above in the definition of inductive limit, then $\lambda$ is injective if and only if $Ker(\lambda_n) \subseteq Ker(\mu_n)$ for all $n \in \N$ and is surjective if and only if $B=\overline{\bigcup_{n=1}^{\infty}\lambda_n( B_n)}$. 

\end{enumerate}

The class of $C^*$-algebras representable as the
inductive limits of sequences of finite-dimensional $C^*$-algebras is known as $AF$-algebras.
Following holds for the inductive limit of $C^*$-algebras:
\begin{theorem}\label{Cindlim}
\begin{enumerate}
\item[(i)]\cite[II.8.2.5]{blackadar2}  The class of simple $C^*$-algebras is closed under inductive limits.
\item[(ii)]  The class of nuclear $C^*$-algebras is closed under inductive limits.
\item[(iii)]\cite[IV.3.4.4]{blackadar2} An inductive limit (with injective connecting maps)
of exact $C^*$-algebras is exact.
\item[(iv)]\cite[II.8.3.24]{blackadar2} Any countable inductive limit of AF algebras is an AF algebra.
\end{enumerate} 
\end{theorem}

\subsubsection{Inductive limits in the category of operator systems}
The existence of inductive limit for the category of operator systems was shown in \cite[$\S 2$]{caralgebra} with connecting morphism as the unital complete order isomorphisms. \\
Let $(\mS,\{\lambda_n\}_{n=1}^{\infty})$ denote the inductive limit of operator systems $\{\mS_n\}$ with $u_n: \mS_n \ra \mS_{n+1}$ the unital complete order ismorphisms, then the unital complete order isomorphisms $\lambda_n : \mS_n \ra \mS$, for
$n=1, 2, \cdots$ are such that (i) $\lambda_n=\lambda_{n+1}\circ u_n$ and (ii) $ \mS=\overline{\bigcup_{n
=1}^{\infty} \lambda_n (\mS_n)}$. Moreover $\mS$
has the universal property that if $\mT$ is an operator system and , $\phi_n : \mS_n \ra \mT$
are unital completely positive maps such that $\phi_n=\phi_{n+1} \circ u_n$ for each $n$, then
there is a unital completely positive map $\phi: \mS \ra \mT$ such that $\phi_n=\phi \circ \lambda_n$.
\vspace{3mm}

The next proposition about the inductive limit of operator systems and their universal $C^*$-covers is well known, and was also used in the Proof of \cite[Proposition 16]{kirch}. For the sake of completeness, we outline the proof.
\begin{proposition}\label{univindlim} $C_u^*(\ilim \mS_n)=\ilim C_u^*(\mS_n),$ for an inducitve sequence $\{\mS_n\}_{n=1}^{\infty}$ with $u_n: \mS_n \ra \mS_{n+1}\; \;\forall n\in \N$ as the connecting complete order isomorphisms.
\end{proposition}
\begin{proof}
Let $(\mS, \{\lambda_n\}_{n=1}^{\infty})$ denote the inductive limit of the inductive sequence $\{\mS_n\}_{n=1}^{\infty}$. Using \Cref{extcover}(i), there exist a unique
$*$-homomorphism $\widetilde{u_n} : C_u^*(\mS_n) \ra C_u^*(\mS_{n+1})$ such that $\widetilde{u_{n}} \circ i_n= i_{n+1} \circ u_{n},$ where $i_n:\mS_n \ra C^*_u(\mS_n)$ denote the natural complete order inclusion for all $n$. Now, the universal property of inductive limits implies that there exist a unital complete order isomorphism $i$ such that the following diagram commutes:
\[\begin{tikzcd}
  \mS_1 \arrow{d}{i_i} \arrow{r}{u_{1}} & \mS_2 \arrow{d}{i_2} \arrow{r}{u_{2}} & \mS_3 \arrow{d}{i_3} \arrow{r}{u_{3}} & \cdots \arrow{r} & \mS \arrow{d}{i} \\ C^*_u(\mS_1) \arrow{r}{\widetilde{u_1} } & C^*_u(\mS_3) \arrow{r}{\widetilde{u_2} } & C^*_u(\mS_3) \arrow{r}{\widetilde{u_3} } & \cdots  \arrow{r} & \displaystyle\ilim(C^*_u(\mS_n))
\end{tikzcd}\]
Clearly, $C^*(i(S))=\ilim C^*_u(\mS_n)$.
Let $\theta: \mS \ra B$ be any other complete order isomorphism. Using the complete order isomorphisms $\theta_n= \theta \circ \lambda_n  : \mS_n \ra B$, there exist surjective $*$-homomorphisms $\widetilde{\theta_n}: C^*_u(\mS_n) \ra C^*(\theta_n(\mS_n)) \subset B$ for all $n\in \N$. Again $\{C^*(\theta_n(\mS_n))\}_{n=1}^{\infty}$ is an inductive sequence with inclusion maps as the connecting maps, and $\ilim C^*(\theta_n(\mS_n)) = C^*(\theta(\mS)) \subset B$. Thus there exist a surjective $*$-homomorphism  $\pi: \ilim  C^*_u(\mS_n) \ra \ilim C^*(\theta_n(\mS_n))=C^*(\theta(\mS))$, and hence 
 $C_u^*(\mS)=\ilim C_u^*(\mS_n).$
\end{proof}
Our aim is to explore the class of operator systems having a particular property under operator system inductive limits. We
present a self-contained treatment here, giving particular emphasis to the role of
the $C^*$-envelopes.


\section[op sys]{Operator system Inductive limits and $C^*$-envelopes}\label{s:2}
From the definition of $C^*$-envelopes, we know that every unital complete order isomorphism
$\phi : \mS \ra \mT$ extends to a unital $C^*$-algebra homomorphism $\pi : C^*_e (\mS)\ra C^*_e (\mT )$ which in general may not be injective, thus unlike universal $C^*$-covers, for the $C^*$-envelopes and inductive sequence $\{\mS_n\}_{n=1}^{\infty}$, $C_e^*(\ilim \mS_n)=\ilim C_e^*(\mS_n)$ is not expected. 


\subsection{$C^*_e$-increasing}
There is no general method to determine the $C^*$-envelopes of even the the lower dimensional operator systems. It is rather strange but amid the list of operator systems whose $C^*$-envelopes are known, majority of operator system pairs are such that the respective pair of $C^*$-envelopes behave as nicely as a pair of universal $C^*$-covers does, that is, for $\mS \subset \mT$ the $C^*$-algebra generated by $\mS$ in $C^*_e(\mT )$
coincides with the $C^*$-envelope of $\mS$, thus giving $C^*_e(\mS) \subset C^*_e(\mT)$:

\vspace{3mm}
\begin{example}[Operator systems associated to discrete groups]\label{ss:1} 
Let $G$ be a countable discrete group,
$\mF$ denote a generating set of $G$ and $\mS(\mF)$ the
operator system associated to $\mF$ by Farenick et al. in
\cite{disgrp}, i.e., $\mS(\mF) := \text{span}\{1, u, u^*:\ u \in \mF\}
\subset C^*(G)$, where $C^*(G)$ denotes the full group $C^*$-algebra
of the group $G$ (\cite[Chapter $8$]{pisier}).  It was shown in
\cite{disgrp} that if $\mF$ is a generating set of the free group
$\mathbb{F}_n$, then $\mS(\mF)$ is independent of the generating set
$\mF$ and is simply denoted by $\mS_n$. Also, $C^*_e(\mS(\mathfrak{u}))=C^*(G)$ (\cite[Proposition 2.2]{disgrp}) and $C^*_e(\mS_r(\mathfrak{u}))=C^*_r(G),$ the reduced group $C^*$-algebra (\cite[Proposition 2.9]{opsysnuc}.  
Recall from \cite[Proposition 2.5.8-2.5.9]{brownozawa}:\\
 If $H$ is a subgroup of a discrete group $G$, then there is a canonical inclusion $$C^*(H) \subset C^*(G).$$ Also, $$C^*_r(H) \subset C^*_r(G)$$ canonically.\\
 Using the fact that the $C^*$-envelope of an operator system associated to the group is the group $C^*$-algebra itself, the preceding statement can be translated in terms of operator systems:\\  \emph{In case $\mathfrak{u}$ and $\mathfrak{v}$ are generating sets of $H \subset G$ and $G$ respectively, then the complete order inclusion $\mS(\mathfrak{u}) \subset \mS(\mathfrak{v})$ can be extended canonically to their $C^*$-envelopes: $C^*_e(\mS(\mathfrak{u})) \subset C^*_e(\mS(\mathfrak{v})).$}
 \end{example}
 
\begin{example}[Graph Operator Systems]\label{ss:1}  
 Given a finite graph $G$ with $n$-vertices, Kavruk et al.~in
\cite{KPTT1} associated an operator system $\mS_G$ as the finite
dimensional operator subsystem of $M_n(\C)$ given by
 $$\mS_G=\mathrm{span} \{\{E_{i,j}: (i,j) \in G\}\cup\{E_{i,
  i}: 1 \leq i \leq n\} \} \subseteq M_n(\C),$$ where $\{E_{i,j}\}$ is
the standard system of matrix units in $M_n(\C)$ and $(i, j)$ denotes
(an unordered) edge in $G$. Ortiz and
Paulsen proved in \cite[Theorem 3.2]{ortiz} that $C^*_e(\mS_G) =
C^*(\mS_G) \subseteq M_n(\C)$ for any graph $G$. Thus, we have: \\
\emph{For graphs $G_1 \subset G_2$, the complete order inclusion $\mS_{G_1} \subset \mS_{G_2}$ extends canonically to a $*$-isomorphism $$C^*_e(\mS_{G_1}) \subset C^*_e(\mS_{G_2})  \subset M_n.$$}
 \end{example}
 \begin{example}[Unital $C^*$-algebras]  Since the $C^*$-envelope of an operator system which is completely order isomorphic to a unital $C^*$-algebra coincides with itself (\cite[Proposition 2.2 (iii)]{opsysnuc}), we naturally have:\\ \emph{If $\mS$ and $\mT$ are both unitally completely order
  isomorphic to unital $C^*$-algebras, then the inclusion $\mS \subset \mT$ extends canonically to $C^*_e(\mS) \subset C^*_e(\mT).$}
  \end{example}
  
 
  \begin{example}[Universal Operator Systems] 
  The universal operator systems (\cite[$\S 1.1$]{kavnuc}) are those for which the universal $C^*$-cover and the $C^*$-envelope coincide. Since for $\mS \subset \mT$, $C^*_u(\mS) \subset C^*_u(\mT)$, trivially:\\ \emph{For a pair of universal operator systems $\mS \subset \mT$, $C^*_e(\mS) \subset C^*_e(\mT).$}
  \end{example} 
  
 \begin{example}[Operator systems for Non-Commuting n-cubes]  Farenick et al. (in \cite{disgrp}) introduced an $(n+1)$-dimensional operator
system  $NC(n)$ as follows:
\vspace*{1mm}

 Let $\mathcal{G} = \{h_1, . . . , h_n\}$, let
$\mathcal{R} = \{h^*_j = h_j , \|h\| \leq 1, 1 \leq j \leq n\}$ be the relation in the set $\mathcal{G}$, and let 
$C^*(\mathcal{G}|\mathcal{R})$ denote the universal unital $C^*$-algebra
generated by $\mathcal{G}$ subject to the relation $\mathcal{R}$. The operator
system $$NC(n) := span\{1, h_1, ..., h_n\} \subset C^*(\mathcal{G}|\mathcal{R})$$
is called \emph{the operator system of the non-commuting $n$-cube.}

They showed that upto a $*$-isomorphism, $C^*_e (NC(n)) =
C^*(*_n\mathbb{Z}_2)$ (\cite[Corollary $5.6$]{disgrp}), so that $*_n \mathbb{Z}_2 \subset *_{n+1} \mathbb{Z}_2$, and hence $C^*(*_n \mathbb{Z}_2) \subset C^*(*_{n+1} \mathbb{Z}_2)$ for all $n$ (as in \Cref{ss:1}), so that:\\
\emph{For $n \in \mathbb{N}$, the complete order inclusion $NC(n) \subset NC(n+1)$ extends to complete order inclusion $C^*_e(NC(n)) \subset C^*_e(NC(n+1))$.}
\end{example} 

\begin{example}[Operator system with simple $C^*$-envelope] 
In case $\mS \subset \mT$ is such that $C^*_e(\mS)$ is simple, we trivially have every homomorphism from $C^*_e(\mS)$ is injective, thus  $C^*_e(\mS) \subset C^*_e(\mT)$. 
\end{example}

\begin{example}[Cuntz Operator System] 

 From \cite{zheng}, for the generators $s_1,s_2,\cdots, s_n$ ($n \geq 2$) of the Cuntz algebra $\mO_n$ and  identity $I$,  the Cuntz operator system $\mS_n$ denotes the operator system generated by $s_1,s_2,\cdots, s_n$, that is,
$$\mS_n = span\{I, s_1,s_2,\cdots, s_n, s_1^*,s_2^*,\cdots, s_n^*\} \subset \mO_n.$$
 Similarly, for the generators  $s_1,s_2,\cdots$ of
$\mO_{\infty}$,$$ \mS_{\infty} = span\{I, s_1,s_2,\cdots, s_1^*,s_2^*,\cdots \} \subset \mO_{\infty}.$$
Also, by \cite{zheng}, for each $n <m$, $\mS_n \subset \mS_m$ and $C^*_e(\mS_n)=\mO_n$ (see also \cite[Proposition $2.8$]{embopsys}). Clearly, $\mO_n \subset \mO_m$.
\end{example}

 For our convenience, we give a name to this property followed by all above examples: 

\begin{definition} We shall call a pair  of operator system $\mS$ and $\mT$ as \emph{$C^*_e$-increasing} if $\mS \subseteq \mT$, then their corresponding $C^*$-envelopes satisfy $C^*_e(\mS)  \subseteq C^*_e(\mT)$. 
\end{definition}


 \subsection{Inductive limit of $C^*_e$-increasing operator systems}
 For an ascending sequence of operator systems, in general, the complete order inclusions need not extend to the complete order inclusions for their sequence of $C^*$-envelopes. But, if each pair in the ascending sequence is $C^*_e$-increasing, we have some ease:
 
\begin{theorem}\label{mainresult} Let $\{\mS_n\}_{n=1}^{\infty}$ be an increasing collection of operator systems such that, for each $n$, $\mS_n \subset \mS_{n+1}$ is $C^*_e$-increasing, then we have
\begin{equation} \label{eq:1} C^*_e(\ilim \mS_n)= \ilim C^*_e(\mS_n).
\end{equation} 
And, moreover if each $\mS_n $ is separable, exact and contains enough unitaries of $C^*_e(\mS_n)$, then $\ilim \mS_n$ embeds into $\mO_2$.
\end{theorem}
\begin{proof}
Let $(\mS, \{\lambda_n\}_{n=1}^{\infty})$ denote the inductive limit of the increasing sequence $\{\mS_n\}_{n=1}^{\infty}$. Using $C^*_e$-increasing nature of each pair, there exist a unique
$*$-homomorphism $\widetilde{u_n} : C_e^*(\mS_n) \ra C_e^*(\mS_{n+1})$ such that $\widetilde{u_{n}} \circ i_n= i_{n+1} \circ u_{n},$ where $i_n:\mS_n \ra C^*_e(\mS_n)$ denotes the natural complete order inclusion for all $n$. Now, the universal property of inductive limits implies that there exist a unital complete order isomorphism $i$ such that the following diagram commutes:
\[\begin{tikzcd}
  \mS_1 \arrow{d}{i_i} \arrow{r}{u_{1}} & \mS_2 \arrow{d}{i_2} \arrow{r}{u_{2}} & \mS_3 \arrow{d}{i_3} \arrow{r}{u_{3}} & \cdots \arrow{r} & \mS \arrow{d}{i} \\ C^*_e(\mS_1) \arrow{r}{\widetilde{u_1} } & C^*_e(\mS_3) \arrow{r}{\widetilde{u_2} } & C^*_e(\mS_3) \arrow{r}{\widetilde{u_3} } & \cdots  \arrow{r} & \displaystyle\ilim(C^*_e(\mS_n))
\end{tikzcd}\]

Clearly, $C^*(i(\mS))= \ilim(C^*_e(\mS_n)).$\\
Claim: $C^*_e(\mS)=C^*(i(\mS))= \ilim (C^*_e(\mS_n)),$ that is, if $\phi: \mS \ra B$ is any other unital complete order isomorphism then there exist a surjective $*$-homomorphism $\pi: C^*(\phi(\mS)) \ra C^*(i(\mS))= \ilim(C^*_e(\mS_n)).$ \\
So, let $\phi: \mS \ra B$ be any unital complete order isomorphism. Then for each $n$, $\phi_n=\phi \circ \lambda_n : \mS_n \ra B$ is a complete order isomorphim. Now, by minimality of $C^*$-envelopes, there exist surjective $*$-homomorphisms $\widetilde{\phi_n}: C^*(\phi_n(\mS_n)) \ra C^*_e(\mS_n)$ for $n \in \N$. Again $\{C^*(\phi_n(\mS_n))\}_{n=1}^{\infty}$ is an inductive sequence with inclusion maps as the connecting morphisms, and $\ilim C^*(\phi_n(\mS_n)) =C^*(\phi(\mS)) \subset B$. So that using the universal property of inductive limits there exist a surjective $*$-homomorphism $\pi:  \ilim C^*(\phi_n(\mS_n)) \ra \ilim  C^*_e(\mS_n)$, and hence $ C^*_e(\mS)=C^*(i(\mS))= \ilim (C^*_e(\mS_n))$. 
By \Cref{opsysenv}, exactness of $\mS_i$ implies exactness of $C^*_e(\mS_i)$, and since inductive limit of exact $C^*$-algebras is exact (\Cref{Cindlim}),  we have $\ilim C^*_e(\mS_i)$ is exact.\\ Thus, using the fact that a separable operator system embeds into $\mO_2$ if and only if it's $C^*$-envelope is exact (\cite[Theorem 3.1]{embopsys}), we have $\ilim \mS_i$ embeds into $\mO_2$.
\end{proof}


Since $\mS_n \subseteq \mS_{n+1}$ extends to their universal $C^*$-covers, $C^*_u(\mS_n) \subseteq C^*_u(\mS_{n+1})$, preceding proof implies the following
\begin{corollary}\label{cor} For a collection $\{\mS_n\}$ of exact universal increasing operator systems with each $\mS_n$ containing enough unitaries of $C^*_e(\mS_n)$, we have $\ilim C^*_e(\mS_n)$ is exact.
\end{corollary}


\begin{corollary} For an increasing collection of operator systems $\{\mS_n\}$, such that for each $n$, the pair $\mS_n \subset \mS_{n+1}$ is $C^*_e$-increasing, and each $C^*_e(\mS_n)$ is an AF-algebra, we have $C^*_e(\ilim \mS_n)$ is also an AF-algebra.
\end{corollary}
\begin{proof} Given such a sequence from \Cref{mainresult},  we have $C^*_e(\ilim \mS_i)=\ilim C^*_e(\mS_i)$, with each $C^*_e(\mS_i)$ an AF-algebra. Thus, $C^*_e(\ilim \mS_i)$ being inductive limit of AF-algebra is again an AF-algebra (\Cref{Cindlim}(iv)).
\end{proof}


\begin{corollary}\label{finiteot} For an increasing collection of operator systems $\{\mS_n\}$  such that either for each $n$, $\mS_n \subseteq C^*_e(\mS_n)$ contain enough unitaries in $C^*_e(\mS_n)$ or $C^*_e(\mS_n)$ is simple for all $n$, we have $C^*_e(\ilim \mS_n \ot_{\min} \mT)=\ilim C^*_e(\mS_n \ot_{\min} \mT)$ and hence is exact, whenever $\mS_n$'s are exact.
\end{corollary}
\begin{proof}
Using \cite[Lemma 2.14]{embopsys}, in either case
upto $*$-isomorphism that fixes $\mS \ot_{\min} \mT$, $C^*_e(\mS \ot_{\min} \mT)= C^*_e(\mS) \ot_{C^*\text{-}\min} C^*_e(\mT)$. Thus, using injectivity of operator system $\mathrm{min}$ tensor product and  $C^*$-$\mathrm{min}$ tensor product, if $\mS_n \subset \mS_{n+1}$ satisfies $C^*_e$-increasing then $\mS_n \ot_{\min} \mT \subset \mS_{n+1} \ot_{\min} \mT$ also satisfies $C^*_e$-increasing. Result now follows by \Cref{mainresult} and \Cref{cor}.
\end{proof}


\subsection{Infinite Tensor Product of operator systems}

Infinite tensor product of a collection $\{A_i : i \in \Omega\}$ of $C^*$-algebras, $\ot_{C^*\text{-}\min} A_i$; ${i \in \Omega}$ has been defined as the inductive limit of the collection $B_{\mathcal{F}}$, where $B_{\mathcal{F}}= A_{i_1} \ot_{C^*\text{-}\min} \cdots \ot_{C^*\text{-}\min} A_{i_n},$ for $\mathcal{F}=\{i_1,\cdots,i_n\}\subseteq \Omega$ (\cite[II.9.8]{blackadar2}). Using the similar technique, Pisier in \cite[Page 390]{pisier} discussed the infinite Haagerup tensor product.\\
 Extending the same to ``$\min$" operator system tensor product, infinite tensor product of a set $\{\mS_i : i \in \Omega\}$ of operator systems, $\ot_{i \in \Omega} \mS_i$, can be defined in terms of inductive limit. In fact, for $\mathcal{F}=\{i_1,\cdots,i_n\}\subseteq \Omega$ , set $$\mathcal{T}_{\mathcal{F}}= \mS_{i_1} \ot_{\min} \cdots \ot_{\min} \mS_{i_n}.$$ Then if $ \mathcal{F} \subset \mathcal{G}$, $\mathcal{T}_{\mathcal{G}} \cong \mathcal{T}_{\mathcal{F}} \ot_{\min}  \mathcal{T}_{\mathcal{G}\setminus \mathcal{F}},$ so that there is a natural inclusion of $\mathcal{T}_{\mathcal{F}}$ into $\mathcal{T}_{\mathcal{G}}$ by $ t \mapsto t \ot 1_{\mathcal{G}\setminus \mathcal{F}}.$  This way, the collection $\mathcal{T}_{\mathcal{G}}$ forms an inductive system and $\ot_{\min} \mS_i;{i \in \Omega}$ is defined to be the inductive limit.

 
\begin{corollary}\label{infiniteot}  For an ascending sequence of operator systems $\{\mS_i\}$  such that for each $i$, $\mS_i \subseteq C^*_e(\mS_i)$ contains enough unitaries in $C^*_e(\mS_i)$ or each $C^*_e(\mS_i)$ is simple, we have
$$C^*_e(\ot_{\min} \mS_i)=\ot_{C^*\text{-}\min} C^*_e(\mS_i).$$
\end{corollary}
\begin{proof} For the finite case $\mS_{i_1} \ot_{\min} \cdots \ot_{\min} \mS_{i_n}$  has enough unitaries of $C^*_e(\mS_{i_1} \ot_{\min} \cdots \ot_{\min} \mS_{i_n})$ and therefore by \Cref{opsysenv}(iv) upto $*$-isomorphism that fixes $\mS_{i_1} \ot_{\min} \cdots \ot_{\min} \mS_{i_n}$;
$$C^*_e(\mS_{i_1} \ot_{\min} \cdots \ot_{min} \mS_{i_n})= C^*_e(\mS_{i_1}) \ot_{C^*\text{-}{\min}} \cdots \ot_{C^*\text{-}{\min}} C^*_e(\mS_{i_n}).$$ Thus, using \Cref{mainresult}, for the infinite tensor product we have:
\begin{eqnarray} C^*_e(\ot_{\min} \mS_i) & = & C^*_e(\ilim  \mS_{i_1} \ot_{\min} \cdots \ot_{\min} \mS_{i_n})\nonumber \\
& =& \ilim C^*_e(\mS_{i_1} \ot_{\min} \cdots \ot_{\min} \mS_{i_n}) \nonumber \\ & = & \ilim C^*_e(\mS_{i_1}) \ot_{C^*\text{-}{\min}} \cdots \ot_{C^*\text{-}{\min}} C^*_e(\mS_{i_n})\nonumber \\ & = &
\ot_{C^*\text{-}{\min}} C^*_e(\mS_i) \nonumber
\end{eqnarray}
\end{proof}


Following is an immediate generalization of Corollary 3.5 of \cite{embopsys} to infinite tensor product using the $C^*$-isomorphism
$\bigotimes_{C^*\text{-}\min} \mathcal{O}_2  \cong \mathcal{O}_2$ (\cite{rordam}). We give a short proof for the sake of completeness:

\begin{corollary}\label{corone} Let $\{\mT_i\}$; $i \in \mathbb{N}$ be a collection of operator systems with separable $C^*$-envelopes. If, for each $i=1,2,\cdots$, $C^*_e(\mT_i)$ is exact, then the operator system $ \ot_{\min} \mT_i$ embeds into $\mO_2$. Converse holds, if either $C^*_e(\mT_i)$ is simple for all $i=1,2,\cdots$ or each $\mT_i$, $i=1,2,\cdots$, contains enough unitaries of $C^*_e(\mT_i)$, respectively.
\end{corollary}
\begin{proof} Let $C^*_e(\mT_i)$ be exact for $ i \in \N.$ Since $\min$ is injective and operator system $\min$ tensor product of $C^*$-algebras embeds complete order isomorphically into their $C^*\text{-}\min$ tensor product \cite[Theorem 4.6, Corollary 4.10]{KPTT1}, using $\bigotimes_{C^*\text{-}\min} \mathcal{O}_2  \cong \mathcal{O}_2$ (\cite{rordam}), we have $$ \bigotimes_{\min} \mT_i  \ \hookrightarrow \bigotimes_{C^*\text{-}\min} \mO_2 \cong \mO_2.$$
Conversely, let there be an embedding of $ \bigotimes_{\min} \mT_i$ into $\mO_2$. In case, $C^*_e(\mT_i)$ is simple for $i \in \N$ or each $\mT_i$, $i\in \N$, contains enough unitaries of $C^*_e(\mT_i)$, respectively, then using \Cref{infiniteot}, $$C^*_e(\ot_{\min} \mT_i)=\ot_{C^*\text{-}\min} C^*_e(\mT_i),$$ which is separable (being the minimal $C^*$-tensor product of separable $C^*$-algebra). Since an operator system is exact if and only if its separable $C^*$-envelope embeds into $\mathcal{O}_2$ \cite[Theorem 3.1]{embopsys}), we have $\ot_{C^*\text{-}\min} C^*_e(\mT_i)$ is exact and hence, for each $i$, the $C^*$-subalgebras $C^*_e(\mT_i)$ is exact (\cite{blackadar2}).
\end{proof}

\subsection{Applications}

\begin{proposition} For an increasing collection $\{\mathfrak{u}_i\}$, where each $\mathfrak{u}_i$ denotes a generating set of the group $G_i$ (\Cref{ss:1}), $C^*_e(\ilim \mS(\mathfrak{u}_i))=\ilim C^*_e(\mS(\mathfrak{u}_i))$ (resp. $C^*_e(\ilim \mS_r(\mathfrak{u}_i))=\ilim C^*_e(\mS_r(\mathfrak{u}_i))$.\\ In particular, for the operator system $\mS_n\subset C^*(\F_n)$ associated to free group $\F_n$ and $\mS_{\infty} \subset C^*(\F_{\infty})$, we have $C^*_e(\mS_{\infty})=C^*(\F_{\infty}).$
\end{proposition}
\begin{proof}
It follows directly by using the fact that whenever $G_i \subseteq G_{i+1}$, $C^*(G_i) \subseteq C^*(G_{i+1})$ and $C^*_r(G_i) \subseteq C^*_r(G_{i+1})$ in \Cref{mainresult}.\\
Also, as the sequence of $\mS_n \subset C^*(\mathbb{F}_n)$ of operator subsystems of full group $C^*$-algebra generated by free group with $n$-generators increases to the operator subsystem $\mS_{\infty} \subset C^*(\F_{\infty}$, therefore we have $C^*_e(\mS_{\infty})=C^*(\F_{\infty}).$
\end{proof}


 \Cref{mainresult} gives an easy proof of the following known result from \cite{zheng}:
\begin{proposition} For the Cuntz operator system of infinite degree, $C^*_e(\mS_{\infty}) = \mO_{\infty}.$
\end{proposition}
\begin{proof} By considering the inductive sequence $\{\mS_n\}_{n=1}^{\infty}$ of Cuntz operator system, since each pair is $C^*_e$-increasing, by \Cref{mainresult}, $C^*_e(\mS_{\infty})=C^*_e(\ilim \mS_n) = \ilim C^*_e(\mS_n)=\ilim \mO_n=\mO_{\infty}.$
\end{proof}


\begin{proposition} For an increasing sequence of graphs as in \Cref{ss:1} , we have $C^*_e(\ilim \mS_{G_i})=C^*(\ilim \mS_{G_i}) \subset B$, where $B$ is an AF $C^*$-algebra.
\end{proposition}
\begin{proof}
Since $C^*_e(\mS_{G_i})=C^*(\mS_{G_i}) \subset M_{n_i}$ (\cite[Thereom 3.2]{ortiz}) and $G_i$'s are increasing, we have $\{M_{n_i}\}$ are also ascending. So that by \Cref{mainresult}, we have $C^*_e(\ilim \mS_{G_i})=\ilim C^*_e(\mS_{G_i})=\ilim C^*(\mS_{G_i}) \subset \ilim  M_{n_i}.$ Using $B= \ilim  M_{n_i}$ which is an AF-algebra, we are done.
\end{proof}


 \begin{remarks}Note that $B$ being nuclear is exact $C^*$-algebra and hence $C^*_e(\ilim \mS_{G_i})$ embeds into $\mO_2$.
 \end{remarks}


\section{Nuclearity Properties of inductive limits in Operator System Category}\label{s:3}
In this section, we study various notions of nuclearity of operator systems.\\
Using injectivity of minimal operator system tensor product (\cite[Theorem 4.6]{KPTT1} and the fact that the maximal commuting operator system tensor product $\mathrm{c}$ is induced by the $\max$ tensor product of universal $C^*$-covers (\cite[Theorem 6.4]{KPTT1}), we have the next proposition:


\begin{proposition}\label{c} For an inductive sequence $\{\mS_n\}_{n=1}^{\infty}$ with $u_n: \mS_n \ra \mS_{n+1}\; \;\forall n\in \N$ as the connecting complete order isomorphisms and any operator system $\mT$, the minimal operator system tensor product $\min$ commutes with the inductive limit, that is, $$\ilim (\mS_n \ot_{\min} \mT )= (\ilim \mS_n) \ot_{\min} \mT.$$ The maximal commuting operator system tensor product $\mathrm{c}$ also commutes with the inductive limit,
$$\ilim (\mS_n \ot_c \mT) = (\ilim \mS_n) \ot_c \mT;$$ provided $\{\mS_n \ot_{\mathrm{c}} \mT\}_{n=1}^{\infty}$ is an inductive system with the complete order isomorphisms $\{u_n \ot Id_{\mT}\}_{n=1}^{\infty}$ as the connecting maps.
\end{proposition}
\begin{proof}
First we give the proof for $\mathrm{c}$ assuming that $\{\mS_n \ot_{\mathrm{c}} \mT\}_{n=1}^{\infty}$ is an inductive system.
The maximal $C^*$-tensor product commutes with arbitrary inductive limits \cite[II.9.6.5]{blackadar2}, so that
 \begin{equation}\label{cmax}
\ilim( C^*_u(\mS_n) \ot_{C^*\text{-}\max} C^*_u(\mT)) = (\ilim C^*_u(\mS_n)) \ot_{C^*\text{-}\max} C^*_u(\mT),
\end{equation}  for the inductive sequence $\{\mS_n\}_{n=1}^{\infty}$ and operator system $\mT$. Also, by \cite[Theorem 6.4]{KPTT1}, the operator system tensor product $\mathrm{c}$ is naturally induced by the maximal operator system tensor product of  universal $C^*$-covers, that is for each $n$, we have $$\mS_n \ot_c \mT \displaystyle \subset C^*_u(\mS_n) \ot_{max} C^*_u(\mT).$$  Therefore, we have the following commutative diagram, 
\[\begin{tikzcd}
  \mS_1 \ot_{\mathrm{c}} \mT \arrow[hookrightarrow]{d}{} \arrow{r}{u_{1} \ot Id_{\mT}} & \mS_2 \ot_{\mathrm{c}} \mT \arrow[hookrightarrow]{d}{} \arrow{r}{u_{2} \ot Id_{\mT}} &  \cdots \arrow{r} & \ilim (\mS_n \ot_{\mathrm{c}} \mT) \arrow[hookrightarrow]{d}{} \\ C^*_u(\mS_1) \ot_{\mathrm{max}} C^*_u(\mT) \arrow{r}{\widetilde{u_1} \ot Id_{C^*_u(\mT)} } & C^*_u(\mS_2) \ot_{\mathrm{max}} C^*_u(\mT) \arrow{r}{\widetilde{u_2} \ot Id_{C^*_u(\mT)} } &  \cdots  \arrow{r} & \displaystyle\ilim(C^*_u(\mS_n) \ot_{\mathrm{max}} C^*_u(\mT))
\end{tikzcd}\]
which further implies, 
\begin{eqnarray}\label{cind}
 \ilim (\mS_n \ot_{\mathrm{c}} \mT)
  & \subset & \ilim (C^*_u(\mS_n) \ot_{\max} C^*_u(\mT))  \nonumber \\
& \subset & \ilim (C^*_u(\mS_n) \ot_{C^*\text{-}\max} C^*_u(\mT)) \text{  (\cite[Theorem 5.12]{KPTT1}}) \nonumber \\ 
& = &  (\ilim C^*_u(\mS_n) ) \ot_{C^*\text{-}\max} C^*_u(\mT) \text{   (\Cref{cmax}) } \nonumber \\
& = & C^*_u(\ilim \mS_n) \ot_{C^*\text{-}\max} C^*_u(\mT)  \text{   (\Cref{univindlim})} \nonumber
\end{eqnarray}
Therefore, using $(\ilim \mS_n )\ot_{\mathrm{c}} \mT \subset C^*_u(\ilim \mS_n) \ot_{\max} C^*_u(\mT)$, the complete order inclusion map above, $\ilim (\mS_n \ot_{\mathrm{c}} \mT) \rightarrow C^*_u(\ilim \mS_n) \ot_{C^*\text{-}\max} C^*_u(\mT) $  induces the complete order isomorphism  $ \ilim (\mS_n \ot_{\mathrm{c}} \mT) = (\ilim \mS_n) \ot_c \mT.$ \\
Similarly, since $\min$ is injective, $$\ilim (\mS_n \ot_{\min} \mT) \subset \ilim (C^*_u(\mS_n) \ot_{\min} C^*_u(\mT)) \subset \ilim (C^*_u(\mS_n) \ot_{C^*\text{-}\min} C^*_u(\mT)).$$ But from  \cite[II.9.6.5]{blackadar2}, $$\ilim (C^*_u(\mS_n) \ot_{C^*\text{-}\min} C^*_u(\mT))=(\ilim C^*_u(\mS_n)) \ot_{C^*\text{-}\min} C^*_u(\mT)=C^*_u(\ilim \mS_n)\ot_{C^*\text{-}\min} C^*_u(\mT). $$ Again, using the injectivity of $\min$, we have $\ilim (\mS_n \ot_{\min} \mT )$ is complete order isomorphic to $(\ilim \mS_n) \ot_{\min} \mT$.
\end{proof}


\subsection{$(\min,\max)$-nuclearity}
It is known that the inductive limit of nuclear $C^*$-algebras (in $C^*$-algebra category) and the inductive limit of nuclear operator spaces (in operator space category) is exact. And, since by \cite[Proposition 5.5]{KPTT2} an operator system is $(\min,\max)$-nuclear (as an operator system) if and
only if it is 1-nuclear as an operator space, we have the inductive limit of $(\min,\max)$-nuclear operator systems is $(\min,\max)$-nuclear:

\begin{proposition}
In case $\{\mS_n\}_{n=1}^{\infty}$ is an inductive system of $(\min,\max)$-nuclear operator system,
then $\ilim \mS_n$ is also $(\min,\max)$-nuclear.
\end{proposition}


\subsection{$(\min,\mathrm{c})$-nuclearity} For finite dimensional operator system $(\min,\mathrm{c})$-nuclearity is the highest nuclearity that can be attained by an operator system which is not a $C^*$-algebra (\cite[Proposition 4.12]{kavnuc}). Next result proves that it is preserved under inductive limit:


\begin{theorem}
In case $\{\mS_n\}_{n=1}^{\infty}$ is an inductive system of $(\min,\mathrm{c})$-nuclear operator system with $u_n: \mS_n \ra \mS_{n+1}\; \;\forall n\in \N$ as the connecting complete order isomorphisms, $\ilim \mS_n$ is also $(\min,\mathrm{c})$-nuclear.
\end{theorem}
\begin{proof} Note that if each operator system $\mS_n$ in the inductive system is $(\min, \mathrm{c})$-nuclear, then using the injectivity of $\min$, the connecting maps $\{u_n \ot Id_{\mT}\}_{n=1}^{\infty}$ are complete order isomorphisms, so that $\{\mS_n \ot_{\mathrm{c}} \mT\}_{n=1}^{\infty}$ forms an inductive system for every operator system $\mT$. Thus result follows directly by \Cref{c} as:  $(\ilim \mS_n) \ot_{\min} \mT = \ilim (\mS_n \ot_{\min} \mT) =\ilim (\mS_n \ot_\mathrm{c} \mT)=(\ilim \mS_n) \ot_{\mathrm{c}} \mT $.
\end{proof}


\subsection{$(\min,\mathrm{el}$)-nuclearity}

It is known that the inductive limit of exact $C^*$-algebras (in $C^*$-algebra category) and the inductive limit of exact operator spaces (in operator space category) is exact. And, since by \cite[Proposition 5.5]{KPTT2} an operator system is exact (as an operator system if and
only if it is exact as an operator space, we have the inductive limit of exact operator systems is exact.
Further we know by \cite{kavnuc}that $(\min,\mathrm{el})$-nuclearity in the operator system category is equivalent to exactness, we have:

\begin{theorem}
If $\mS_n$ is an inductive system of $(\min,\mathrm{el})$-nuclear operator system, $\ilim \mS_n$ is also  $(\min,\mathrm{el})$-nuclear.
\end{theorem}


\subsection{$(\min,\mathrm{ess})$-nuclearity}
In an attempt to study nuclearity of operator systems through their $C^*$-envelopes, it was proved in \cite[Proposition 4.2.]{opsysnuc} that an operator system is  $(\min,\mathrm{ess})$-nuclear if its $C^*$-envelope is nuclear. Moreover, if an operator system contains enough
unitaries of its $C^*$-envelope, then its  $(\min,\mathrm{ess})$-nuclearity is equivalent to the nuclearity
of its $C^*$-envelope. We make use of this result to prove  $(\min,\mathrm{ess})$-nuclearity of an inductive system of  $(\min,\mathrm{ess})$-nuclear operator systems.


\begin{theorem}\label{mainresult2} Let $\{\mS_n\}_{n=1}^{\infty}$ be an increasing collection of operator systems such that, for each $n$, $\mS_n \subseteq C^*_e(\mS_n)$ contain enough unitaries in $C^*_e(\mS_n)$. Moreover, assume that each pair $\mS_n \subseteq \mS_{n+1}$ is $C^*_e$-increasing. If each $\mS_n$ is $(\min,\mathrm{ess})$-nuclear, then $\ilim \mS_n$ is  $(\min,\mathrm{ess})$-nuclear.
\end{theorem}
\begin{proof}
From \Cref{mainresult}, $\ilim C^*_e(\mS_n) =C^*_e(\ilim \mS_n)$.  Now using \cite[Proposition 4.2.]{opsysnuc},  $(\min,\mathrm{ess})$-nuclearity of $\mS_n$ implies $C^*$-nuclearity of $C^*_e(\mS_n)$, and since inductive limit of nuclear $C^*$-algebras is nuclear (\Cref{Cindlim}(ii)),  we have $\ilim C^*_e(\mS_n)$ and hence $C^*_e(\ilim \mS_n)$ is nuclear. Therefore, $ \ilim \mS_n$ is  $(\min,\mathrm{ess})$-nuclear (\cite[Proposition 4.2.]{opsysnuc}).
\end{proof}


\subsection{$(\mathrm{el,c})$-nuclearity}
It is known from Theorem 7.1 and  Theorem 7.6 of \cite{KPTT2} (see also \cite[Theorem 4.7]{kavnuc}) that an operator system $\mS$ is $(\mathrm{el,c})$-nuclear if and only if it has DCEP which is true, if and only if $\mS \ot_{\min} C^*(\mathbb{F}_{\infty}) = \mS \ot_{\max} C^*(\mathbb{F}_{\infty}).$ This characterization is useful to study at $(\mathrm{el,c})$-nuclearity of inductive limit of operator systems with $(\mathrm{el,c})$-nuclearity:


\begin{theorem}
Let $\{\mS_n\}_{n=1}^{\infty}$ be an inductive sequence of $(\mathrm{el},\mathrm{c})$-nuclear operator systems (equivalently, with DCEP), then $\ilim \mS_n$ is also $(\mathrm{el},\mathrm{c})$-nuclear (equivalently, has DCEP).
\end{theorem}
\begin{proof}  Since for each $n$, $\mS_n$ is $(\mathrm{el},\mathrm{c})$-nuclear, we have \begin{equation}
\mS_n \ot_{\max} C^*(\mathbb{F}_{\infty})= \mS_n \ot_{\min}  C^*(\mathbb{F}_{\infty}). \nonumber
\end{equation}
Also, $\mathrm{el}$ is left injective (\cite[Theorem 7.5]{KPTT1}), therefore for each $n$, $u_n \ot Id_{\mT}$ is a complete order isomorphism, so that  $\{\mS_n \ot_{\mathrm{c}} \mT\}_{n=1}^{\infty}$ forms an inductive system with $\{u_n \ot Id_{\mT}\}_{n=1}^{\infty}$ as the connecting maps.
Using this and the fact that maximal commuting tensor product $\mathrm{c}$ coincides with the operator system $\max$ tensor product if one of the tensorial factor is a $C^*$-algebra (\cite[Theorem 6.7]{KPTT1}), we have
 \begin{eqnarray}
(\ilim \mS_n) \ot_{\min} C^*(\mathbb{F}_{\infty}) & =& \ilim (\mS_n \ot_{\min} C^*(\mathbb{F}_{\infty})) \text{  (\Cref{c}}) \nonumber \\
& = & \ilim (\mS_n \ot_{\max} C^*(\mathbb{F}_{\infty}))\nonumber \\
& = & \ilim (\mS_n \ot_{\mathrm{c}} C^*(\mathbb{F}_{\infty}))\nonumber \\
& = & (\ilim \mS_n) \ot_{\mathrm{c}}  C^*(\mathbb{F}_{\infty}) \text{  (\Cref{c}}) \nonumber \\
& = & (\ilim \mS_n) \ot_{\max}  C^*(\mathbb{F}_{\infty}) \nonumber
\end{eqnarray}
Hence, $\ilim \mS_n$ has DCEP provided each $\mS_n$ has DCEP.
\end{proof}


\subsection*{Acknowledgements}
{\small The authors would like to thank Ved Prakash Gupta for his careful reading of the manuscript and helpful comments.}

\bibliographystyle{plain}
\bibliography{ind_lim_ref}
\end{document}